\documentclass[12pt, reqno]{amsart}
\usepackage{amsmath, amsthm, amscd, amsfonts, amssymb, graphicx, color}
\usepackage{bm}
\usepackage[all,cmtip]{xy}
\usepackage{multirow,bigdelim}

\newtheorem{theorem}{Theorem}[section]
\newtheorem{lemma}[theorem]{Lemma}

\theoremstyle{definition}

\numberwithin{equation}{section}

\newcommand{\vp}{\varphi}

\newcommand{\clb}{\mathcal{B}}

\newcommand{\cle}{\mathcal{E}}

\newcommand{\clh}{\mathcal{H}}
\newcommand{\clk}{\mathcal{K}}

\newcommand{\clm}{\mathcal{M}}

\newcommand{\clq}{\mathcal{Q}}

\newcommand{\cls}{\mathcal{S}}

\newcommand{\z}{\bm{z}}
\newcommand{\w}{\bm{w}}
\newcommand{\K}{\bm{k}}

\newcommand{\D}{\mathbb{D}}
\newcommand{\B}{\mathbb{B}}

\newcommand{\raro}{\rightarrow}

\begin{document}

%\today

\setcounter{page}{1}

%\today

\title[Commutant lifting and interpolation]{Commutant lifting and Nevanlinna-Pick interpolation in several variables}

\author[Deepak]{Deepak K. D.}
\address{Indian Statistical Institute, Statistics and Mathematics Unit, 8th Mile, Mysore Road, Bangalore, 560059,
India}
\email{dpk.dkd@gmail.com }

\author[Pradhan]{Deepak Kumar Pradhan}
\address{Indian Statistical Institute, Statistics and Mathematics Unit, 8th Mile, Mysore Road, Bangalore, 560059,
India}
\email{deepak12pradhan@gmail.com}

\author[Sarkar]{Jaydeb Sarkar}
\address{Indian Statistical Institute, Statistics and Mathematics Unit, 8th Mile, Mysore Road, Bangalore, 560059,
India}
\email{jay@isibang.ac.in, jaydeb@gmail.com}

\author[Timotin]{Dan Timotin}
\address{Institute of Mathematics of the Romanian Academy, P.O. Box 1-764, Bucharest, 014700, Romania}
\email{Dan.Timotin@imar.ro}



\begin{abstract}
This paper concerns a commutant lifting theorem and a Nevanlinna-Pick type interpolation result in the setting of multipliers from vector-valued Drury-Arveson space to a large class of vector-valued reproducing kernel Hilbert spaces over the unit ball in $\mathbb{C}^n$. The special case of reproducing kernel Hilbert spaces includes all natural examples of Hilbert spaces like Hardy space, Bergman space and weighted Bergman spaces over the unit ball.
\end{abstract}

\maketitle

\section{Introduction}

Let $\D = \{z \in \mathbb{C} : \displaystyle |z| < 1\}$ be the open unit disc in the complex plane $\mathbb{C}$. The classical Nevanlinna–Pick interpolation theorem \cite{RN, Pick} states: Given distinct $n$ points $\{z_i\}_{i=1}^n \subseteq \D$ (initial data) and $n$ points $\{w_i\}_{i=1}^n \subseteq \D$ (target data), there exists a $\vp \in H^\infty(\D)$ such that $\|\vp\|_{\infty} \leq 1$ and such that
\[
\varphi(z_{i}) = w_{i},
\]
for all $i = 1, \ldots, n$, if and only if the \textit{Pick matrix}
\[
\begin{bmatrix}
\frac{1 - w_i \bar{w}_j}{1 - z_i \bar{z}_j}\end{bmatrix}_{i,j=1}^n,
\]
is positive semi-definite. Here we denote by $H^\infty(\D)$ the Banach algebra of all bounded analytic functions on $\D$ equipped with the norm $\|\vp\| = sup\{|\vp(z)|: z \in \D\}$, $\vp \in H^\infty(\D)$. In his seminal paper \cite{Sarason}, Sarason proved the commutant lifting theorem for compressions of the shift operator to shift co-invariant subspaces of the Hardy space which gives a simpler and elegant proof of the Nevanlinna–Pick interpolation theorem.

Sarason's approach to the commutant lifting theorem, along with its direct application to Nevanlinna–Pick interpolation theorem, is deeply connected with a number of classical problems in function theory and operator theory and have been studied extensively in the past few decades (cf. \cite{F}). There also has been a great deal of interest in analyzing the possibilities of commutant lifting theorem and interpolation (and other related problems) in the setting of general reporducing kernel Hilbert spaces over domains in $\mathbb{C}^n$, $n \geq 1$ (for instance, see \cite{BB}, \cite{EP}, \cite{AT}, \cite{BTV} and \cite{DL}).

In this paper we make a contribution to a commutant lifting theorem and a version of Nevanlinna-Pick interpolation interpolation in several variables. To be more precise, let $m \geq 1$ and let $\clh_m$ denotes the reproducing kernel Hilbert space corresponding to the kernel $k_m$ on $\mathbb{B}^n$, where
\[
k_m(\z, \w) = (1 - \sum_{i=1}^n z_i \bar{w}_i)^{-m} \quad \quad (\z, \w \in \mathbb{B}^n),
\]
and $\mathbb{B}^n = \{\z = (z_1, \ldots, z_n) \in \mathbb{C}^n: \displaystyle \sum_{i=1}^n |z_i|^2 < 1\}$. Recall that $\clh_m$ is the Drury-Arveson space (popularly denoted by $H^2_n$), the Hardy space, the Bergman space and the weighted Bergman space over $\mathbb{B}^n$ for $m = 1$, $m = n$, $m = n+1$ and $m > n+1$, respectively.

Our main results, restricted to $\clh_m$, $m > 1$, can now be formulated as follows:

\noindent \textit{Commutant lifting theorem (Theorem \ref{thm-CLT main}):} Suppose $\clq_1$ and $\clq_2$ are joint $(M_{z_1}, \ldots, M_{z_n})$ co-invariant subspaces of $H^2_n (= \clh_1)$ and $\clh_m$, respectively. Let $X \in \clb(\clq_1, \clq_2)$ and $\|X\| \leq 1$. If
\[
X (P_{\clq_1} M_{z_i}|_{\clq_1}) = (P_{\clq_2} M_{z_i}|_{\clq_2}) X,
\]
for all $i = 1, \ldots, n$, then there exists a holomorphic function $\vp : \mathbb{B}^n \raro \mathbb{C}$ such that the multiplication operator $M_{\vp} \in \clb(H^2_n, \clh_m)$, $\|M_{\vp}\| \leq 1$ (that is, $\vp$ is a \textit{contractive multiplier}), and
\[
X = P_{\clq_2} M_{\vp}|_{\clq_1}.
\]
Thus, we have the following commutative diagram:
\[
\xymatrix{
H^2_n \ar@{.>}[rr]^{\displaystyle M_{\vp}} \ar@{->}[dd]_{\displaystyle P_{\clq_1}}
&& \clh_m \ar@{->}[dd]^{\displaystyle P_{\clq_2}}    \\ \\
\clq_1 \ar@{->}[rr]_{\displaystyle X} && \clq_2
}
\]
Given a closed subspace $\cls$ of a Hilbert space $\clh$ we denote by $P_{\cls}$ the orthogonal projection of $\cls$ on $\clh$.

\noindent \textit{Nevanlinna–Pick interpolation theorem (Theorem \ref{thm-interpolation}):} Given distinct $n$ points $\{\z_i\}_{i=1}^n \subseteq \mathbb{B}^n$ and $n$ points $\{w_i\}_{i=1}^n \subseteq \D$, there exists a contractive multiplier $\vp$ such that
\[
\varphi(\z_{i}) = w_{i},
\]
for all $i = 1, \ldots, n$ if and only if the matrix
\[
\begin{bmatrix}
\frac{1}{(1 - \langle \z_i, {\z}_j \rangle)^{m}} - \frac{w_i \bar{w}_j}{1 - \langle \z_i, {\z}_j \rangle}
\end{bmatrix}_{i,j=1}^n,
\]
is positive semi-definite. Here $\langle \z, \w \rangle = \displaystyle \sum_{i=1}^n z_i \bar{w}_i$ for all $\z, \w \in \mathbb{C}^n$.

We make strong use of the commutant lifting theorem in the setting of Drury-Arveson space (see Theorem \ref{thm-BTV}) and a refined factorization result (see Theorem \ref{multiplier-factor}) concerning multipliers between Drury-Arveson space and a large class of analytic reproducing kernel Hilbert space over $\mathbb{B}^n$.

The remainder of the paper is organized as follows. Section 2 discusses some useful and known facts about reproducing kernel Hilbert spaces. Section 3 presents the commutant lifting theorem. Section 4 is devoted to factorizations of multipliers. The factorization results obtained here may be of independent interest. Section 5 provides the interpolation theorem.

\section{Preliminaries}

The Drury-Arveson space over the unit ball $\mathbb{B}^n$ in $\mathbb{C}^n$ will be denoted by $H^2_n$. Recall that $H^2_n$ is a reproducing kernel Hilbert space corresponding to the kernel function
\[
k_1(\z, \w) = (1 - \displaystyle \sum_{i=1}^n z_i \bar{w}_i)^{-1} \quad \quad (\z, \w \in \mathbb{B}^n).
\]
Let $k : \mathbb{B}^n \times \mathbb{B}^n \raro \mathbb{C}$ be a kernel such that $k$ is analytic in the first variables $\{z_1, \ldots, z_n\}$. We say that $k$ is \textit{regular} if there exists a kernel $\tilde{k} : \mathbb{B}^n \times \mathbb{B}^n \raro \mathbb{C}$, analytic in $\{z_1, \ldots, z_n\}$, such that
\[
k(\z, \w) = k_1(\z, \w) \tilde{k}(\z, \w) \quad \quad (\z, \w \in \mathbb{B}^n).
\]
If $k$ is a regular kernel, then $\clh_k$, the reproducing kernel Hilbert space corresponding to the kernel $k$, will be referred as a \textit{regular reproducing kernel Hilbert space}.

In the case of a regular reproducing kernel Hilbert space $\clh_k$, it follows \cite{KSST} that $M_{z_i}$, the multiplication operator by the coordinate function $z_i$, is bounded. Note that
\[
(M_{z_i} f)(\w) = w_i f(\w),
\]
for all $f \in \clh_k$, $\w \in \mathbb{B}^n$ and $i = 1, \ldots, n$. Moreover, it also follows that the commuting tuple $(M_{z_1}, \ldots, M_{z_n})$ on $\clh_k$ is a \textit{row contraction}, that is
\[
\sum_{i=1}^n M_{z_i} M_{z_i}^* \leq I_{\clh_k}.
\]
If $\cle$ is a Hilbert space, then we also say that $\clh_k \otimes \cle$ is a regular reproducing kernel Hilbert space. Note that the kernel function of $\clh_k \otimes \cle$ is given by
\[
\mathbb{B}^n \times \mathbb{B}^n \ni(\z, \w) \mapsto k(\z, \w) I_{\cle}.
\]
The $\cle$-valued Drury-Arveson space, denoted by $H^2_n(\cle)$, is the reproducing kernel Hilbert space corresponding to the $\clb(\cle)$-valued kernel function
\[
\B^n \times \B^n \ni (\z, \w) \mapsto k_1(\z, \w) I_{\cle}.
\]
To simplify the notation, we often identify $H^2_n(\cle)$ with $H^2_n \otimes \cle$ via the unitary map defined by $z^{\bm{k}} \eta \mapsto z^{\bm{k}} \otimes \eta$ for all $\K \in \mathbb{Z}_+^n$ and $\eta\in \cle$. This also enable us to identify $(M_{z_1}, \ldots, M_{z_n})$ on $H^2_n(\cle)$ with $(M_{z_1} \otimes I_{\cle}, \ldots, M_{z_n} \otimes I_{\cle})$ on $H^2_n \otimes \cle$.

Typical examples of regular reproducing kernel Hilbert spaces arise from weighted Bergman spaces over $\mathbb{B}^n$. More specifically, let $\lambda >1$, and let
\begin{equation}\label{eq-lambda}
k_{\lambda}(\z, \w) = (1 - \displaystyle \sum_{i=1}^n z_i \bar{w}_i)^{-\lambda} \quad \quad (\z, \w \in \mathbb{B}^n).
\end{equation}
Then $\clh_{k_{\lambda}}$ is a regular reproducing kernel Hilbert space. Note that $\clh_{k_{\lambda}}$ is the Hardy space, Bergman space and weighted Bergman space for $\lambda = n$, $n+1$ and $n+1+\alpha$ for any $\alpha > 0$, respectively.

Suppose $\clh$ and $\cle_*$ are Hilbert spaces and $(T_1, \ldots, T_n)$ is a commuting tuple of bounded linear operators on $\clh$. We say that $(T_1, \ldots, T_n)$ on $\clh$ \textit{dilates} to $(M_{z_1} \otimes I_{\cle_*}, \ldots, M_{z_n} \otimes I_{\cle_*})$ on $H^2_n \otimes \cle_*$ if there exists an isometry $\Pi : \clh \raro H^2_n \otimes \cle_*$ such that
\[
\Pi T_i^* = (M_{z_i} \otimes I_{\cle_*})^* \Pi,
\]
for all $i = 1, \ldots, n$ (cf. \cite{JS2}). We often say that $\Pi : \clh \raro H^2_n \otimes \cle_*$ is a \textit{dilation} of $(T_1, \ldots, T_n)$.

If $\clh = \clh_k$ is a regular reproducing kernel Hilbert space, then by [Theorem 6.1, \cite{KSST}], it follows that $(M_{z_1} \otimes I_{\cle}, \ldots, M_{z_n} \otimes I_{\cle})$ on $\clh_k \otimes \cle$ dilates to $(M_{z_1} \otimes I_{\cle_*}, \ldots, M_{z_n} \otimes I_{\cle_*})$ on $H^2_n \otimes \cle_*$ for some Hilbert space $\cle_*$. More specifically:

\begin{theorem}\label{thm-KSST}
Let $\cle$ be a Hilbert space. If $\clh_k$ is a regular reproducing kernel Hilbert space, then there exist a Hilbert space $\cle_*$ and an isometry
\[
\Pi_k: \clh_k \otimes \cle \raro H^2_n \otimes \cle_*,
\]
such that
\[
\Pi_k (M_{z_i} \otimes I_{\cle})^* = (M_{z_i} \otimes I_{\cle_*})^* \Pi_k,
\]
for all $i = 1, \ldots, n$.
\end{theorem}

Since $(M_{z_1} \otimes I_{\cle}, \ldots, M_{z_n} \otimes I_{\cle})$ on $\clh_k \otimes \cle$ is a pure row contraction \cite{KSST}, the above result also directly follows from Muller-Vasilescu \cite{MV} and Arveson \cite{WA}.

In what follows, given a Hilbert space $\clh$ and a closed subspace $\clq$ of $\clh$, we will denote by $i_{\clq}$ the inclusion map
\[
i_{\clq} : \clq \hookrightarrow \clh.
\]
Note that $i_{\clq}$ is an isometry and
\[
i_{\clq} i_{\clq}^* = P_{\clq}.
\]

We now recall the commutant lifting theorem in the setting of the Drury-Arveson space (see \cite{AT} or Theorem 5.1, page 118, \cite{BTV}).
A closed subspace $\clq$ of a regular reproducing kernel Hilbert space $\clh_k \otimes \cle$ is said to be \textit{shift co-invariant} if
\[
(M_{z_i} \otimes I_{\cle})^* \clq \subseteq \clq \quad \quad (i=1, \ldots, n).
\]

\begin{theorem}\label{thm-BTV}
Let $\cle_1$ and $\cle_2$ be Hilbert spaces. Suppose $\clq_1$ and $\clq_2$ are shift co-invariant subspaces of $H^2_n(\cle_1)$ and $H^2_n(\cle_2)$, respectively, $X \in \clb(\clq_1, \clq_2)$ and let $\|X\| \leq 1$. If
\[
X (P_{\clq_1} M_{z_i}|_{\clq_1}) = (P_{\clq_2} M_{z_i}|_{\clq_2}) X,
\]
for all $i = 1, \ldots, n$, then there exists a multiplier $\Phi \in \clm(H^2_n(\cle_1), H^2_n(\cle_2))$ such that $\|M_{\Phi}\| \leq 1$ and $P_{\clq_2} M_{\Phi}|_{\clq_1} = X$.
\end{theorem}

Recall also that, given regular reproducing kernel Hilbert spaces $\clh_{k_1} \otimes \cle_1$ and $\clh_{k_2} \otimes \cle_2$, a function $\Phi : \mathbb{B}^n \raro \clb(\cle_1, \cle_2)$ is called a \textit{multiplier} from $\clh_{k_1} \otimes \cle_1$ to $\clh_{k_2} \otimes \cle_2$ if
\[
\Phi (\clh_{k_1} \otimes \cle_1) \subseteq \clh_{k_2} \otimes \cle_2.
\]
The \textit{multiplier space} $\clm(\clh_{k_1} \otimes \cle_1, \clh_{k_2} \otimes \cle_2)$ is the set of all multipliers from $\clh_{k_1} \otimes \cle_1$ to $\clh_{k_2} \otimes \cle_2$. In what follows, $\clm_1(H^2_n \otimes \cle_1, \clh_k \otimes \cle_2)$ will denote the closed ball of radius one:
\[
\clm_1(H^2_n \otimes \cle_1, \clh_k \otimes \cle_2) = \{\Phi \in \clm(H^2_n \otimes \cle_1, \clh_k \otimes \cle_2): \|M_{\Phi}\| \leq 1\}.
\]
We have the following useful characterization of multipliers (cf. Proposition 4.2, \cite{JS2}): Let $\clh_k$ be a regular reproducing kernel Hilbert space, and let $X \in \clb(H^2_n \otimes \cle_1, \clh_k \otimes \cle_2)$. Then
\[
X (M_{z_i} \otimes I_{\cle_1}) = (M_{z_i} \otimes I_{\cle_2}) X,
\]
if and only if $X = M_{\Phi}$ for some $\Phi \in \clm(H^2_n \otimes \cle_1, \clh_k \otimes \cle_2)$.

\section{Commutant lifting theorem }

We begin with a general result concerning intertwiner of bounded linear operators.

\begin{lemma}\label{lemma-thm-BTV variation}
Suppose $\Pi : \clh \raro \clk$ and $\hat{\Pi} : \hat{\clh} \raro \hat{\clk}$ are isometries, $V\ \in \clb(\clk)$, $\hat{V} \in \clb(\hat{\clk})$, $T = \Pi^* V\Pi$ and $\hat{T} = \hat{\Pi}^* \hat{V} \hat{\Pi}$. Moreover, let $X \in \clb(\clh, \hat{\clh})$ satisfies
\[
XT = \hat{T} X.
\]
If we define
\[
\clq = \Pi \clh \quad \mbox{and} \quad \hat{\clq} = \hat{\Pi} \hat{\clh},
\]
and
\[
\tilde{X} = \hat{\Pi} X \Pi^*|_{\clq},
\]
then $\tilde{X} \in \clb(\clq, \hat{\clq})$ and
\[
\tilde{X} (P_{\clq} V|_{\clq}) = (P_{\hat{\clq}} \hat{V}|_{\hat{\clq}}) \tilde{X}.
\]
\end{lemma}
\noindent \textsf{Proof.}
Notice that $P_{\clq} = \Pi \Pi^*$ and $P_{\hat{\clq}} = \hat{\Pi} \hat{\Pi}^*$. Hence
\[
\tilde{X} = (\hat{\Pi} \hat{\Pi}^*) \hat{\Pi} X \Pi^*|_{\clq} = P_{\hat{\clq}} (\hat{\Pi} X \Pi^*)|_{{\clq}},
\]
and in particular
\[
(\hat{\Pi} X \Pi^*) {\clq} \subseteq \hat{\clq},
\]
which shows that $\tilde{X} \in \clb(\clq, \hat{\clq})$. Moreover
\[
\begin{split}
\tilde{X} (P_{\clq} V|_{\clq}) & = \hat{\Pi} X \Pi^* P_{\clq} V|_{\clq}
\\
& = \hat{\Pi} X \Pi^* V|_{\clq}
\\
& = \hat{\Pi} X T \Pi^*|_{\clq}
\\
& = \hat{\Pi} \hat{T} X \Pi^*|_{\clq}
\\
& = \hat{\Pi} \hat{\Pi}^* \hat{V} \hat{\Pi} (\hat{\Pi}^* \hat{\Pi}) X \Pi^*|_{\clq}
\\
& = P_{\hat{\clq}} \hat{V}|_{\hat{\clq}} \hat{\Pi} X \Pi^*|_{\clq}
\\
& = (P_{\hat{\clq}} \hat{V}|_{\hat{\clq}})\tilde{X},
\end{split}
\]
which completes the proof.
\qed

Now we are ready to prove a variation, in terms of dilations, of Theorem \ref{thm-BTV}.

\begin{theorem}\label{thm-new BTV}
Let $\clh$ and $\hat{\clh}$ be Hilbert spaces. Suppose $T = (T_1, \ldots, T_n)$ and $\hat{T} = (\hat{T}_1, \ldots, \hat{T}_n)$ are commuting tuples on $\clh$ and $\hat{\clh}$, respectively, $X \in \clb(\clh, \hat{\clh})$, $\|X\| \leq 1$, and
\[
XT_i = \hat{T}_i X,
\]
for all $i = 1, \ldots, n$. If $\Pi : \clh \raro H^2_n \otimes \cle$ and $\hat{\Pi} : \hat{\clh} \raro H^2_n \otimes \hat{\cle}$ are dilations of $T$ and $\hat{T}$, respectively, then there exists a multiplier $\Phi \in \clm_1(H^2_n \otimes \cle, H^2_n \otimes \hat{\cle})$ such that
\[
{X} = \hat{\Pi}^* M_{\Phi}\Pi.
\]
\end{theorem}
\begin{proof} Let
\[
\clq = \Pi \clh \quad \mbox{and} \quad \hat{\clq} = \hat{\Pi} \hat{\clh}.
\]
If
\[
\tilde{X} = \hat{\Pi} X \Pi^*|_{\clq},
\]
then by Lemma \ref{lemma-thm-BTV variation}, it follows that $\tilde{X} \in \clb(\clq, \hat{\clq})$ and
\[
\tilde{X} (P_{\clq} (M_{z_i} \otimes I_{\cle})|_{\clq}) = (P_{\hat{\clq}} (M_{z_i} \otimes I_{\hat{\cle}})|_{\hat{\clq}}) \tilde{X},
\]
for all $i = 1, \ldots, n$. It then follows from the commutant lifting theorem, Theorem \ref{thm-BTV}, that
\[
\tilde{X} = P_{\hat{\clq}} M_{\Phi}|_{\clq},
\]
for some $\Phi \in \clm(H^2_n \otimes \cle, H^2_n \otimes \hat{\cle})$ and $\|M_{\Phi}\| \leq 1$. Then
\[
\hat{\Pi} X \Pi^*|_{\clq} = P_{\hat{\clq}} M_{\Phi}|_{\clq}.
\]
It then follows from
\[
\clq = \mbox{ran~} \Pi = \mbox{ran~} \Pi \Pi^*,
\]
that
\[
(\hat{\Pi} X \Pi^*) (\Pi \Pi^*) = P_{\hat{\clq}} M_{\Phi} (\Pi \Pi^*).
\]
Thus
\[
\hat{\Pi} X = P_{\hat{\clq}} M_{\Phi} \Pi = (\hat{\Pi} \hat{\Pi}^*) M_{\Phi} \Pi,
\]
and hence ${X} = \hat{\Pi}^* M_{\Phi} \Pi$.
\end{proof}

Now let $\clq$ be a shift co-invariant subspace of $\clh_k \otimes \cle$. An isometry $\Pi : \clq \raro H^2_n \otimes \cle_*$ is said to be a \textit{dilation of $\clq$} if
\[
\Pi (P_{\clq} (M_{z_i} \otimes I_{\cle})|_{\clq})^* = (M_{z_i} \otimes I_{\cle_*})^* \Pi,
\]for all $i = 1, \ldots, n$, that is $(P_{\clq}M_{z_1}|_{\clq}, \ldots, P_{\clq}M_{z_n}|_{\clq})$ on $\clq$ dilates to $(M_{z_1} \otimes I_{\cle_*}, \ldots, M_{z_n} \otimes I_{\cle_*})$ on $H^2_n \otimes \cle_*$ via the isometry $\Pi$.

\begin{lemma}\label{lemma-PiQ}
Let $\clh_k$ be a regular reproducing kernel Hilbert space, and let $\cle$ and $\cle_*$ be a Hilbert spaces. Suppose $\clq$ is a shift co-invariant subspace of $\clh_k \otimes \cle$. If $\Pi : \clh_k \otimes \cle \raro H^2_n \otimes \cle_*$ is a dilation of $\clh_k \otimes \cle$, then $\Pi_{\clq} : \clq \raro H^2_n \otimes \cle_*$, defined by
\[
\Pi_{\clq} = \Pi \circ i_{\clq},
\]
is a dilation $\clq$.
\end{lemma}
\begin{proof} We first observe that
\[
\Pi_{\clq}^* \Pi_{\clq} = i_{\clq}^* \Pi^*  \Pi i_{\clq} = I_{\clq}.
\]
Now we compute
\[
\begin{split}
\Pi_{\clq} (P_{\clq}  (M_{z_i} \otimes I_{\cle})|_{\clq})^* & = \Pi i_{\clq} P_{\clq}  (M_{z_i} \otimes I_{\cle})^*|_{\clq}
\\
& = \Pi (M_{z_i} \otimes I_{\cle})^*|_{\clq}
\\
& =  (M_{z_i} \otimes I_{\cle_*})^* \Pi|_{\clq}
\\
& =  (M_{z_i} \otimes I_{\cle_*})^* (\Pi i_{\clq}) i_{\clq}^*|_{\clq}
\\
& =  (M_{z_i} \otimes I_{\cle_*})^* \Pi_{\clq} i_{\clq}^*|_{\clq}.
\end{split}
\]
Now
\[
i_{\clq}^*|_{\clq} = I_{\clq},
\]
and so
\[
\Pi_{\clq} (P_{\clq}  (M_{z_i} \otimes I_{\cle})|_{\clq})^* = (M_{z_i} \otimes I_{\cle_*})^* \Pi_{\clq},
\]
for all $i = 1, \ldots, n$. This completes the proof of the lemma.
\end{proof}

We are now ready to present and prove the commutant lifting theorem.

\begin{theorem}\label{thm-CLT main}
Let $\clh_k$ be a regular reproducing kernel Hilbert space, $\cle_1$ and $\cle_2$ be Hilbert spaces, and let $\clq_1$ and $\clq_2$ be shift co-invariant subspaces of $H^2_n \otimes \cle_1$ and $\clh_k \otimes \cle_2$, respectively. Let $X \in \clb(\clq_1, \clq_2)$, and suppose that $\|X\| \leq 1$ and
\[
X (P_{\clq_1} (M_{z_i} \otimes I_{\cle_1})|_{\clq_1}) = (P_{\clq_2} (M_{z_i} \otimes I_{\cle_2})|_{\clq_2}) X,
\]
for all $i = 1, \ldots, n$. Then there exists a multiplier $\Phi \in \clm_1(H^2_n \otimes \cle_1, \clh_k \otimes \cle_2)$ such that
\[
X = P_{\clq_2} M_{\Phi}|_{\clq_1}.
\]
\end{theorem}

\begin{proof}
Observe that the inclusion map $i_{\clq_1} : \clq_1 \hookrightarrow H^2_n \otimes \cle_1$ is a dilation of $\clq_1$. Let $\Pi_k : \clh_k \otimes \cle_2 \raro H^2_n \otimes \hat{\cle}$ be a dilation of $\clh_k$ (see Theorem \ref{thm-KSST}), that is, $\Pi_k$ is an isometry and
\begin{equation}\label{eq:added1}
\Pi_k (M_{z_i} \otimes I_{\cle_2})^* = (M_{z_i} \otimes I_{\hat{\cle}})^* \Pi_k,
\end{equation}
for all $i = 1, \ldots, n$ and some Hilbert space $\hat{\cle}$. Set
\[
\Pi_{\clq_2} = \Pi_k i_{\clq_2}.
\]
By Lemma \ref{lemma-PiQ}, it follows that $\Pi_{\clq_2} : \clq_2 \raro H^2_n \otimes \hat{\cle}$ is a dilation of $\clq_2$. Then Theorem \ref{thm-new BTV} yields
\[
X = \Pi_{\clq_2}^* M_{\Phi_1} i_{\clq_1},
\]
for some multiplier $\Phi_1 \in \clm(H^2_n \otimes \cle_1, H^2_n \otimes \hat{\cle})$. Hence
\[
X = i_{\clq_2}^* (\Pi_k^* M_{\Phi_1}) i_{\clq_1}.
\]
Since
\[
M_{\Phi_1} (M_{z_i} \otimes I_{\cle_1})=(M_{z_i} \otimes I_{\hat{\cle}})M_{\Phi_1},
\]
we have, using also the adjoint of~\eqref{eq:added1},
\[
\Pi_k^* M_{\Phi_1} (M_{z_i} \otimes I_{\cle_1})
= \Pi_k^*  (M_{z_i} \otimes I_{\hat{\cle}})M_{\Phi_1}
= (M_{z_i} \otimes I_{\cle_2}) \Pi_k^* M_{\Phi_1},
\]
for all $i = 1, \ldots, n$, that is, $\Pi_k^* M_{\Phi_1} : H^2_n \otimes \cle_1 \raro \clh_k \otimes \cle_2$ intertwines the shifts. Consequently
\[
\Pi_k^* M_{\Phi_1} = M_{\Phi},
\]
for some multiplier $\Phi \in \clm(H^2_n \otimes \cle_1, \clh_k \otimes \cle_2)$. Hence
\[
X = i_{\clq_2}^* M_{\Phi} i_{\clq_1},
\]
and thus
\[
i_{\clq_2} X  = P_{\clq_2} M_{\Phi} i_{\clq_1}.
\]
Hence, we have
\[
X = P_{\clq_2} M_{\Phi}|_{\clq_1}.
\]
Finally
\[
\|M_{\Phi}\| \leq \|M_{\Phi_1}\| \leq 1,
\]
completes the proof of the theorem.
\end{proof}

A simpler way of presenting the above theorem, from Hilbert module point of view, is to say that the following diagram commutes:

\[
\xymatrix{ & &  & H^2_n \otimes \hat{\mathcal  E}  \\
& H^2_n \otimes \mathcal E_1 \ar[d]_{P_{\mathcal Q_1}} \ar@{.>}[r]_{M_{\Phi}} \ar@/^/@{.>}[urr]^{M_{\Phi_1}} & \mathcal H_k\otimes \mathcal E_2 \ar[d]_{P_{\mathcal Q_2}} \ar[ur]^{\Pi_k} & \\
& \mathcal Q_1 \ar[r]_{X}  & \mathcal Q_2  \ar@/_/[uur]_{\Pi_{\mathcal Q_2}} }
\]

\section{Factorizations}

Let $k$ be a regular reproducing kernel on $\mathbb{B}^n$. Then there exists a positive definite kernel $\tilde{k} : \mathbb{B}^n \times \mathbb{B}^n \raro \mathbb{C}$ such that
\[
k(\z, \w) = k_1(\z, \w) \tilde{k}(\z, \w)  \quad \quad (\z, \w \in \mathbb{B}^n).
\]
Let $\clh_{\tilde{k}}$ be the reproducing kernel Hilbert space corresponding to the kernel $\tilde{k}$. Suppose $\w \in \mathbb{B}^n$ and $ev(\w) : \clh_{\tilde{k}} \raro \mathbb{C}$ is the evaluation map, that is
\[
ev(\w) (f) = f(\w) \quad \quad (f \in \clh_{\tilde{k}}).
\]
Then
\[
\tilde{k} (\z, \w) = ev(\z)  ev(\w)^* \quad \quad (\z, \w \in \mathbb{B}^n),
\]
and so
\begin{equation}\label{eq:product formula}
k(\z,\w) =  k_1(\z, \w) \Big(ev(\z)  ev(\w)^*\Big) \quad \quad (\z,\w \in \mathbb B^n).
\end{equation}
From  Corollary 4.2 in~\cite{KSST} it follows that the map
\[
(\pi F)(\z):=F(\z, \z)  \quad \quad (\z \in \mathbb{B}^n),
\]
defines a coisometry from $H^2_n  \otimes\clh_{\tilde k}=\clh_{k_1}\otimes\clh_{\tilde k}$ to $\clh_{k}=\clh_{k_1\tilde k}$. If we view $H^2_n  \otimes\clh_{\tilde k}$ as a reproducing kernel Hilbert space of functions with values in $\clh_{\tilde k}$, then the map $ \pi $ is actually the multiplier $ M_{ev} $; indeed, if we compute the action on reproducing kernels, we have
\[
M_{ev}(f\otimes g)(\w)= f(\w)\otimes ev(\w)(g)=f(\w)\otimes g(\w)=
\pi(f\otimes g)(\w).
\]

This formula may be extended by tensorizing with $ I_\cle $, where $ \cle  $ is a Hilbert space. If we define $ \Psi_k:\clh_{\tilde{k}} \otimes \cle\to\cle $ by $ \Psi_k:=ev\otimes I_\cle $, then $ \Psi_k $ is obviously also a coisometric multiplier.
Taking into account~\eqref{eq:product formula}, we obtain the following theorem (see also \cite[Theorem 4.1]{CH} and \cite[Theorem 6.2]{KSST}):

\begin{theorem}\label{kernel-factor}
	Let $k : \mathbb{B}^n \times \mathbb{B}^n \raro \mathbb{C}$ be a regular kernel, and let
	\[
	k(\z, \w) = k_1(\z, \w) \tilde{k}(\z, \w) \quad \quad (z, \w \in \mathbb{B}^n),
	\]
	for some kernel $\tilde{k}$ on $\mathbb{B}^n$. Suppose $\clh_{\tilde{k}}$ is the reproducing kernel Hilbert space corresponding to the kernel $\tilde{k}$. If $\cle$ is a Hilbert space, then there exists a co-isometric multiplier $\Psi_k \in \clm(H^2_n \otimes (\clh_{\tilde{k}} \otimes \cle), \clh_k \otimes \cle)$ such that
	\[
	k(\z,\w) I_{\mathcal E} = \frac{ \Psi_k(\z)  \Psi_k(\w)^*}{1-\langle \z, \w \rangle} \quad \quad (\z,\w \in \mathbb B^n).
	\]
\end{theorem}
%{
%\color{red}
%There is something I do not get here. If we apply the theorem to $ \cle=\mathbb C $, we obtain
%\[
%k(\z,\w)=\Psi_k(\z)  \Psi_k(\w)^* k_1(\z, \w).
%\]

%So
%\begin{equation}\label{eq:added2}
%\tilde{k}(\z, \w)= \Psi_k(\z)  \Psi_k(\w)^* .
%\end{equation}
%This last relation  does not seem to have anything with the Drury--Arveson space---so it follows that for any kernel $ \tilde{k} $ there exists a co-isometric multiplier $\Psi_k \in \clm(H^2_n \otimes \clh_{\tilde{k}} , \clh_{\tilde kk_1} )$ such that~\eqref{eq:added2} is valid.

%Is this correct? If it is, I feel again that there should be a connection with Section 4 from~\cite{KSST}}

In particular, it is instructive to consider the familiar case: weighted Bergman spaces over $\mathbb{B}^n$. Let $m >1$ be an integer and let
\[
k_m(\z, \w) = (1 - \displaystyle \sum_{i=1}^n z_i \bar{w}_i)^{-m} \quad \quad (\z, \w \in \mathbb{B}^n).
\]
Then
\[
\tilde{k}_m (\z, \w) = k_{m-1}(\z, \w),
\]
and hence $\Psi_{k_m}(\w)^* : \cle \raro \clh_{k_{m-1}} \otimes \cle$ is given by
\[
\Psi_{k_m}(\w)^* \eta = k_{m-1}(\cdot, \w) \otimes \eta,
\]
for all $\z, \w \in \mathbb{B}^n$ and $\eta \in \cle$. Note also that
\[
\langle \Psi_{k_m}(\w) (f \otimes \eta), \zeta \rangle = f(\w) \langle \eta, \zeta \rangle,
\]
for all $f \in \clh_{k_{m-1}}$, $\eta, \zeta \in \cle$ and $\w \in \mathbb{B}^n$.

\noindent For this particular case, the representation of $\Psi_{k_m}$ has been computed explicitly in Section 4 of \cite{BDS} and \cite{BB}.

Now suppose $\cle_1$ and $\cle_2$ are Hilbert spaces, and $k$ is a regular kernel on $\mathbb{B}^n$. Let $\Theta : \mathbb{B}^n \raro \clb(\cle_1, \cle_2)$ be an analytic function. From~\cite[Theorem 6.28]{PR} it follows that $\Theta \in \clm_1(H^2_n \otimes \cle_1, \clh_k \otimes \cle_2)$ if and only if
\[
%\sum_{1 \leq i, j \leq m} \langle (I - M_{\Theta} %M^*_{\Theta}) k(\w_i, \w_j) \eta_j, \eta_i \rangle \geq 0,
k(\z,\w)-k_1(\z,\w)\Theta(\z)\Theta(\w)^*
\]
is a positive definite kernel.
By virtue of Theorem \ref{kernel-factor}, this is equivalent to positive definiteness of the kernel
\[
(\z, \w) \mapsto k_1(\z, \w) (\Psi_k(\z) \Psi_k(\w)^* - \Theta(\z)\Theta(\w)^*).
\]
We may then apply~\cite[Theorem 8.57($(i)\Rightarrow(ii)$)]{AMC} to obtain the following theorem.

\begin{theorem}\label{multiplier-factor}
	Let $\cle_1$ and $\cle_2$ be Hilbert spaces, and let $\Theta : \mathbb{B}^n \raro \clb(\cle_1, \cle_2)$ be an analytic function. In the setting of Theorem \ref{kernel-factor}, the following conditions are equivalent:
	
	(i) $\Theta \in \clm_1(H^2_n \otimes \cle_1, \mathcal H_k \otimes \cle_2)$,

	(ii) there exists $\tilde{\Theta } \in \clm_1(H^2_n \otimes \cle_1, H^2_n \otimes (\clh_{\tilde{k}} \otimes \cle_2))$ such that
	\[
	M_{\Theta}  = M_{\Psi_k} M_{\tilde{\Theta}}.
	\]
\end{theorem}

More specifically, the multiplier $\Psi_k$ makes the following diagram commutative:

\setlength{\unitlength}{3mm}
\begin{center}
	\begin{picture}(40,16)(0,0)
	\put(12.7,3){$H^2_n \otimes \cle_1$}\put(19,1.6){$M_{\Theta}$}
	\put(22.9,3){$\clh_k \otimes \cle_2$} \put(22, 10){$H^2_n \otimes (\clh_{\tilde{k}} \otimes \cle_2)$} \put(24,9.2){ \vector(0,-1){5}} \put(15.8,
	4.3){\vector(1,1){5.5}} \put(18.4,
	3.4){\vector(1,0){4}}\put(15.8,6.8){$M_{\tilde{\Theta}}$}\put(25.3,6.5){$M_{\Psi_k}$}
	\end{picture}
\end{center}

One should compare Theorems \ref{kernel-factor} and \ref{multiplier-factor} with Lemma 4.1 and Theorem 4.2 in \cite{BDS} and Theorem 2.1 in \cite{BB}.

\section{Nevanlinna-Pick interpolation}

We now turn to the interpolation problem. Let $\cle_1$ and $\cle_2$ be Hilbert spaces. We denote by $\clb_1(\cle_1, \cle_2)$ the open unit ball of $\clb(\cle_1, \cle_2)$, that is
\[
\clb_1(\cle_1, \cle_2) = \{ A \in \clb(\cle_1, \cle_2): \|A\| < 1\}.
\]
We aim to solve the following version of Pick-type interpolation problem: Suppose $\{\z_i\}_{i=1}^m \subseteq \mathbb{B}^n$, $\{W_i\}_{i=1}^m \subseteq  \mathcal B_1(\mathcal E_1, \mathcal E_2)$ and $m \geq 1$. Find necessary and sufficient conditions (on $\{\z_i\}_{i=1}^m$ and $\{W_i\}_{i=1}^m$) for the existence of a multiplier $\Phi \in \clm_1 (H^2_n \otimes \cle_1, \clh_k \otimes \cle_2)$ such that
\begin{equation}\label{interpolation-condition}
\Phi(\z_{i}) = W_{i},
\end{equation}
for all $i = 1, \ldots, m$.

Given such data $\{\z_i\}_{i=1}^m \subseteq \mathbb{B}^n$, $\{W_i\}_{i=1}^m \subseteq  \mathcal B_1(\mathcal E_1, \mathcal E_2)$, set
\[
\clq_1 = \{\sum_{i=1}^m k_1(\cdot, \z_i) \zeta_i : \zeta_i \in \cle_1\},
\]
and
\[
\clq_2 = \{\sum_{i=1}^m k(\cdot, \z_i) \eta_i : \eta_i \in \cle_2\}.
\]
Obviously $\clq_1$ and $\clq_2$ are shift co-invariant subspaces of $H^2_n \otimes \cle_1$ and $\clh_k \otimes \cle_2$, respectively. Define $X: \clq_2 \rightarrow \clq_1$ by
\[
X k(\cdot, \z_i) \eta = k_{1}(\cdot, \z_i) (W_{i}^* \eta),
\]
for all $i = 1, \ldots, m$ and $\eta \in \cle_2$. Then
\[
X (M_{z_i} \otimes I_{\cle_2})^*|_{\clq_2} = (M_{z_i} \otimes I_{\cle_1})^*|_{\clq_1} X,
\]
for all $i = 1, \ldots, n$. Then, by Theorem \ref{thm-CLT main}, $X$ is a contraction if and only if there exists $\Phi\in  \clm_1( H^2_n \otimes \cle_1, \clh_k \otimes \cle_2)$ such that
\[
P_{\clq_2}M_{\Phi}|_{\clq_1} = X^*.
\]
In particular
\[
\begin{split}
k_{1} (\cdot,\z_i) (W_{i}^* \eta) & = X (k(\cdot, \z_i) \eta)
\\
& = M^*_{\Phi} (k(\cdot,\z_i) \eta)
\\
& = k_{1}(\cdot, \z_i) (\Phi(\z_i)^* \eta),
\end{split}
\]
for all $\eta \in \cle_2$ and $i = 1, \ldots, m$, and so $\Phi$ satisfies \eqref{interpolation-condition}. Conversely, if $\Phi$ satisfies \eqref{interpolation-condition}, then it is easy to see that $X$ defines a contractions from $\clq_2$ to $\clq_1$.

\noindent Now $X$ is a contraction if and only if
\[
\begin{split}
0 & \leq \langle (I - X^* X)\sum_{i=1}^m k(\cdot,\z_i) \eta_i, \sum_{i=1}^m k(\cdot,\z_i) \eta_i \rangle
\\
& \Rightarrow \sum_{1 \leq i, j \leq m}\langle k(\z_i,\z_j) \eta_j, \eta_i \rangle - \sum_{1 \leq i, j \leq m} \langle W_{i} k_1(\z_i,\z_{j}) W^*_{j} \eta_j, \eta_i \rangle \geq 0
\\
& \Rightarrow \sum_{1 \leq i, j \leq m} \langle \Big(k(\z_i, \z_j) I_{\cle_2} - \frac{W_i W^*_j}{1-\langle \z_{i}, \z_{j} \rangle}\Big)  \eta_j, \eta_i \rangle \geq 0,
\end{split}
\]
for all $\eta_1, \ldots, \eta_m \in \cle_2$, where the last equality follows from Theorem \ref{kernel-factor}.

\noindent On the other hand, Theorem \ref{multiplier-factor} says that $\Phi \in \clm_1(H^2_n \otimes \cle_1, \clh_k \otimes \cle_2)$ if and only if there exists $\tilde{\Phi} \in \clm_1(H^2_n \otimes \cle_1, H^2_n \otimes (\clh_{\tilde{k}} \otimes \cle_2))$ such that
\[
\Phi(\z) = \Psi_k(\z) \tilde{\Phi}(\z),
\]
for all $\z \in \mathbb{B}^n$. Summarizing, we have established the following interpolation theorem:

\begin{theorem}\label{thm-interpolation}
Let $\cle_1$ and $\cle_2$ be Hilbert spaces, $k$ be a regular kernel on $\mathbb{B}^n$, and let
\[
k(\z, \w) = k_1(\z, \w) \tilde{k}(\z, \w) \quad \quad (\z, \w \in \mathbb{B}^n),
\]
for some kernel $\tilde{k}$ on $\mathbb{B}^n$. Suppose $\{\z_i\}_{i=1}^m \subseteq \mathbb{B}^n$ and $\{W_i\}_{i=1}^m \subseteq  \mathcal B_1(\mathcal E_1, \mathcal E_2)$. Then the following conditions are equivalent:

(i) There exists a multiplier $\Phi \in \clm_1 (H^2_n \otimes \cle_1, \clh_k \otimes \cle_2)$ such that $\Phi(\z_{i}) = W_{i}$ for all $i = 1, \ldots, m$.

(ii) $\displaystyle \sum_{1 \leq i, j \leq m} \langle \Big(k(\z_i, \z_j) I_{\cle_2} - \frac{W_i W^*_j}{1-\langle \z_{i}, \z_{j} \rangle}\Big)  \eta_j, \eta_i \rangle$ for all $\eta_1, \ldots, \eta_m \in \cle_2$.

(iii) There exists a multiplier $\tilde{\Phi} \in \clm_1(H^2_n \otimes \cle_1, H^2_n \otimes (\clh_{\tilde{k}} \otimes \cle_2))$ such that
\[
\Psi_k(\z_i) \tilde{\Phi}(\z_i) = W_i \quad \quad (i=1, \ldots, n).
\]
\end{theorem}

If $n=1$ and $\tilde{k}(z, w) = (1 - z \bar{w})^{-m}$, $m \in \mathbb{N}$ (that is, weighted Bergman space over $\D$ with an integer weight), then a part of Theorem \ref{thm-interpolation} was proved by Ball and Bolotnikov \cite{BB}.

Note that, the positivity condition in part (ii) of Theorem \ref{thm-interpolation} does not hold in general:

\noindent \textsf{Example:} Consider the regular kernel $k$ as the Bergman kernel on $\D$, that is
\[
k(z,w) = \frac{1}{(1-z\bar{w})^2} \quad \quad (z, w \in \D).
\]
Here
\[
{k}(z, w) = \tilde{k}(z, w) = \Psi_k(z) \Psi_k^*(w)= \frac{1}{(1-z\bar{w})}\quad \quad (z, w \in \D).
\]
Then, for a given pair of points $\{w_1,w_2\} \subseteq \D$, condition (ii) in Theorem \ref{thm-interpolation} holds  for some pair $\{z_1, z_2\} \subseteq \D$ if and only if
\[
\begin{bmatrix}
&\frac{1}{1-|z_1|^2} -|w_1|^2   & \frac{1}{1- z_1 \bar{z}_2} - w_1 \bar{w}_2\\
& \frac{1}{1-z_2 \bar{z}_1} - w_2 \bar{w}_1 & \frac{1}{1-|z_2|^2} - |w_2|^2
\end{bmatrix} \diamond
\begin{bmatrix}
&\frac{1}{1-|z_1|^2}     & \frac{1}{1- z_1\bar{z}_2} \\
& \frac{1}{1-z_2\bar{z}_1}   & \frac{1}{1-|z_2|^2}
\end{bmatrix} \geq 0,
\]
where $`\diamond $' denotes the Schur product of matrices. However, if $z_1 = w_2 = 0$ and $z_2 \neq 0$, then it is easy to see that the positivity condition fails to hold for any $w_1 \in \D$ such that
\[
\frac{1 - |w_1|^2}{1 - |z_2|^2} < 1.
\]

%\vspace{0.1in}

\noindent\textbf{Acknowledgement:}
The research of the second named author is supported by NBHM  (National Board of Higher Mathematics, India) post-doctoral fellowship no: 0204/27-\\/2019/R\&D-II/12966.
The research of the third named author is supported in part by NBHM grant NBHM/R.P.64/2014, and the Mathematical Research Impact Centric Support (MATRICS) grant, File No: MTR/2017/000522 and Core Research Grant, File No: CRG/2019/000908, by the Science and Engineering Research Board (SERB), Department of Science \& Technology (DST), Government of India. The third author also would like to thanks the Institute of Mathematics of the Romanian Academy, Bucharest, Romania, for its warm hospitality during a visit in May 2019.

\end{document}